\documentclass[12pt]{amsart}
\usepackage[utf8]{inputenc}
\usepackage[english]{babel}

\usepackage[left=25mm, right=25mm, top=25mm, bottom=25mm]{geometry}
\usepackage[shortlabels]{enumitem}
\usepackage{polynom, subcaption, hyperref, color, comment}
\hypersetup{colorlinks, citecolor=black, filecolor=black, linkcolor=black, urlcolor=black}
\usepackage{soul}
\usepackage[colorinlistoftodos]{todonotes}

\usepackage{graphicx, float, tikz-cd, tikz, caption}
\usepackage{csquotes}

\makeatletter
\def\Ddots{\mathinner{\mkern1mu\raise\p@
\vbox{\kern7\p@\hbox{.}}\mkern2mu
\raise4\p@\hbox{.}\mkern2mu\raise7\p@\hbox{.}\mkern1mu}}
\makeatother

\usepackage{amsmath, amsfonts, amssymb, amsthm, mathtools, mathrsfs, commath}

\newtheorem {theorem}{Theorem}[section]
\newtheorem {lemma}[theorem]{Lemma}
\newtheorem {proposition}[theorem]{Proposition}
\newtheorem {conjecture}[theorem]{Conjecture}

\newenvironment{theorembis}[1]
  {\renewcommand{\thetheorem}{\ref{#1}$'$}%
   \addtocounter{theorem}{-1}%
   \begin{theorem}}
  {\end{theorem}}
\theoremstyle{definition}
\newtheorem {definition}[theorem]{Definition}
\newtheorem {remark}[theorem]{Remark}

\newtheorem {example}[theorem]{Example}
\newtheorem*{theorem*}{Theorem}

\newtheorem{thm}{Theorem}

\DeclareMathOperator\curl{curl}
\DeclareMathOperator\Hess{Hess}

\begin{document}

% Titlepage
\title{From $2N$ to infinitely many escape orbits}
% Authors
\author{Josep Fontana-McNally}
\thanks{Josep Fontana-McNally was supported by an INIREC grant Introduction to research financed under the  project “Computational, dynamical and geometrical complexity in fluid dynamics”, Ayudas Fundación BBVA a Proyectos de Investigación Científica 2021. Josep Fontana, Eva Miranda and Cédric Oms are partially supported by the Spanish State Research Agency grant PID2019-103849GB-I00 of AEI / 10.13039/501100011033 and by the AGAUR project 2021 SGR 00603.}
\address{{Laboratory of Geometry and Dynamical Systems, Department of Mathematics}, Universitat Polit\`{e}cnica de Catalunya and Department of Mathematics,  University of Toronto}
\email{josep.fontana@estudiantat.upc.edu}
\author{ Eva Miranda}
\thanks{ Eva Miranda  is supported by the Catalan Institution for Research and Advanced Studies via an ICREA Academia Prize 2021 and by the Alexander Von Humboldt foundation via a Friedrich Wilhelm Bessel Research Award. Eva Miranda is also supported by the Spanish State
Research Agency, through the Severo Ochoa and Mar\'{\i}a de Maeztu Program for Centers and Units
of Excellence in R\&D (project CEX2020-001084-M).  Eva Miranda and Daniel Peralta-Salas  acknowledge partial support from the grant
“Computational, dynamical and geometrical complexity in fluid dynamics”, Ayudas Fundación BBVA a Proyectos de Investigación Científica 2021.  }
\address{{Laboratory of Geometry and Dynamical Systems \& IMTech, Department of Mathematics}, Universitat Polit\`{e}cnica de Catalunya and CRM, Barcelona, Spain \\ Centre de Recerca Matemàtica-CRM}
\email{eva.miranda@upc.edu}
\author{Cédric Oms}
\thanks{Cédric Oms acknowledges financial support from the Margarita Salas postdoctoral contract 
financed by the European Union-NextGenerationEU and is partially supported by the ANR grant ``Cosy" (ANR-21-CE40-0002), partially supported by the ANR grant ``CoSyDy" (ANR-CE40-0014).}
\address{{Laboratory of Geometry and Dynamical Systems \& IMTech, Department of Mathematics}, Universitat Polit\`{e}cnica de Catalunya and BCAM Bilbao, Mazarredo, 14. 48009 Bilbao Basque Country - Spain}
\email{coms@bcamath.org}
\author{Daniel Peralta-Salas}
\address{Instituto de Ciencias Matemáticas (ICMAT), Consejo Superior de Investigaciones Científicas, Madrid}
\email{dperalta@icmat.es}
\thanks{Daniel Peralta-Salas is supported by the grants CEX2019-000904-S and PID2019-106715GB GB-C21 funded by
MCIN/AEI/10.13039/501100011033.}
\dedicatory{To Alain Chenciner with admiration on his 2N-birthday}

\maketitle

% ABSTRACT
\begin{abstract}
     In this short note, we prove that singular Reeb vector fields associated with generic $b$-contact forms have either (at least) $2N$ or an infinite number of escape orbits, where $N$ denotes the number of connected components of the critical set. %This result is to be compared with the longstanding conjecture that under mild conditions the Reeb vector field of a smooth contact manifold should possess either two or an infinite number of periodic orbits (proved for non-degenerate periodic orbits in \cite{CDR22}).
   %  We obtain this result as a corollary of the same statement for the number of escape orbits of singular Beltrami fields using the singular version of Etnyre-Ghrist's contact/Beltrami correspondence.
\end{abstract}

%%%%%%%%%%%%%%%%%%%%%%%%%%%%%%%%%%%%%%%%%%%%%%%%%%%%%%%
%SECTIONS OF THE ARTICLE

\section{Introduction}

The article \cite{MOP22} delved into the  dynamical behavior of $b$-Beltrami vector fields on $b$-manifolds  of dimension 3.
There, the authors examined the presence of escape orbits for $b$-Beltrami vector fields, which are orbits whose $\alpha$- or $\omega$-limit set is a point on the critical set of the underlying $b$-structure. Uhlenbeck's celebrated theorem on generic eigenfunctions of the Laplacian was key to demonstrating that for generic asymptotically exact $b$-metrics (Definition \ref{def:asymptoticallyexact}), $b$-Beltrami vector fields have escape orbits. A straightforward application of the singular version of Etnyre-Ghrist's correspondence between $b$-Reeb and $b$-Beltrami vector fields (Theorem \ref{thm:contactbeltrami}) then led to a similar result for Melrose $b$-contact forms (Definition \ref{def:melrosebcontact}) and the corresponding $b$-Reeb vector fields. 

In this article we build upon those results and show, through a refined analysis of \cite{MOP22}, that the number of escape orbits of a generic $b$-Reeb field -- not just those associated to Melrose $b$-contact forms -- or any $b$-Beltrami field for a generic asymptotically exact $b$-metric is either (at least) $2N$ or infinite.
%More concretely, the existence of escape orbits for those vector fields was studied, which are orbits whose $\alpha$- or $\omega$-limit set is a point on the critical set. The analysis carried out in \cite{MOP22} confirms that these vector fields admit escape orbits for a vast majority of asymptotically exact $b$-metrics by utilizing the singular counterpart of Etnyre-Ghrist's contact/Beltrami correspondence and results by Uhlenbeck regarding eigenfunctions of the Laplacian. In this article, we refine those results and prove that the number of escape orbits is either two-$N$ or infinity, by refining the analysis carried out in \cite{MOP22}.
By interpreting escape orbits as semi-orbits of singular periodic orbits as in \cite{MO21}, this result, in the context of singular contact geometry, is consistent with the long-standing conjecture of two or an infinite number of periodic orbits for Reeb vector fields.

In 1979, Weinstein put forward a conjecture stating that a Reeb vector field on a closed 3-manifold would always have at least one periodic orbit \cite{W79}. After steady progress towards the proof of this conjecture, it was proved in full generality in dimension $3$ by Taubes using Seiberg-Witten Floer homology \cite{T05}.
Further advancements were made by Cristofaro-Gardiner and Hutchings \cite{CH}, who proved that every Reeb vector field on a closed 3-manifold has at least two periodic orbits. Currently, it is conjectured that a Reeb vector field will have either two or infinitely many periodic orbits. The existence of infinitely many periodic orbits has been established under certain assumptions (see the survey \cite{GG}). Generically, this is known to be true, see \cite{irie}. Recently, Colin, Dehornoy, and Rechtman \cite{CDR22} proved that for non-degenerate Reeb orbits the number of periodic orbits is either two or infinity.

In the singular context, the study of $b$-Reeb vector fields was initiated in \cite{MO21}, where some cases of a singular version of the Weinstein conjecture were proved. Further investigation of this conjecture was pursued  in \cite{PW} and \cite{VW}.

In \cite{MOP22} a semi-local variant of this hypothesis was examined proving that escape orbits exist. However, the authors did not address the count of these orbits. In this paper we establish that the number of such orbits is either (at least) $2N$ or infinite, where $N$ is the number of connected components of the critical set. For the proof we use a more elementary approach that does not rely on the singular Reeb-Beltrami correspondence. 
%It is worth noting that, rather than requiring the non-degeneracy condition on the orbits, we impose a genericity condition on a specific metric (see Section \ref{sec:bBeltrami}). Loosely, our result is the following (see Theorem \ref{thm:main} for the precise formulation).

%A semi-local version of this conjecture was studied in \cite{MOP22}, where as mentioned before, the existence of escape orbits was proved. The authors left out a discussion on the number of such orbits. We prove in this note that the number is either $2N$ or $\infty$, where $N$ is the number of connected components of the critical set.
%Observe that instead of the condition on non-degeneracy on the orbits, we ask a genericity condition on a certain metric (see Section \ref{sec:bBeltrami}). 

%Our results suggest that the  \emph{"two or infinite number of periodic Reeb orbits"} conjecture holds for singular contact manifolds as well, and not only for a generic class.
%Our result points towards the direction that a version of the ``two or infinite number of periodic Reeb orbits" conjecture holds also for singular contact manifolds (and not only for a generic class).

\begin{thm}
Let $Z$ be a compact embedded surface in a $3$-dimensional manifold $M$. Then for a generic $b$-contact form having $Z$ as critical set, the associated $b$-Reeb vector field has at least $2N$ escape orbits, and infinitely many if the first Betti number of $Z$ is positive.
\end{thm}

This result improves the main result contained in \cite{MOP22} in several ways: It does not only improve the lower bound on the number of escape orbits, but also greatly broadens the notion of \emph{genericity}. While in \cite{MOP22} the main result assumes that the $b$-contact forms are defined using an auxiliary so-called asymptotically exact $b$-metric, in the above result a $C^\infty$-small perturbation in the space of $b$-forms is sufficient to prove this lower bound.

However, when dealing with $b$-Beltrami fields that are not $b$-Reeb vector fields, we still require the genericity spectral arguments from \cite{MOP22}. In the final section, we examine the count of escape orbits in this specific scenario of $b$-Beltrami fields.

\begin{thm}
    Let $Z$ be a compact embedded surface in a $3$-dimensional manifold $M$. Then for a generic asymptotically exact $b$-metric on $(M,Z)$, any $b$-Beltrami vector field has at least $2N$ escape orbits, and infinitely many if the first Betti number of $Z$ is positive.
\end{thm}
\subsection*{Acknowledgments}\textcolor{black}{The authors are indebted to the valuable comments of the anonymous referees, that improved substantially the results, proofs and the presentation of the previous version of this paper.}
%%%%%%%%%%%%%%%%%%%%%%%%%%%%%%%%%%%%%%
\section{Preliminaries}

In this section we give a brief introduction to $b$-contact geometry and outline how it naturally arises in certain singularities of celestial mechanics. For the proof of the second main theorem (concerning $b$-Beltrami fields), we will also need to review the connections between $b$-Reeb and $b$-Beltrami fields, which we introduce in Section \ref{sec:bBeltrami}. 

\subsection{$b$-contact geometry}

The notion of $b$-manifolds was introduced by Melrose in \cite{M93}, and later expanded upon by Nest and Tsygan in the context of deformation quantization on manifolds with boundary \cite{NT96}. More recently, a systematic study of $b$-symplectic manifolds was carried out in \cite{GMP14}, which led to an increased interest in various aspects of such structures \cite{GL13,C13}. The language of $b$-manifolds is useful whenever one encounters logarithmic singularities in differential forms along a smooth hypersurface, which could be a boundary. These appear naturally, for example, when compactifying manifolds with cylindrical ends. In particular, $b^m$-symplectic and $b^m$-contact forms arise when applying the McGehee transformation to study double collisions and behaviour near infinity of the restricted planar three-body problem, as shown in \cite{invitation, MO21}. A natural question to ask is how different the dynamics of contact forms and $b$-contact forms are, leading to a formulation of a ``singular'' Weinstein conjecture in \cite{MO21}, which we introduce in Section \ref{sec:singularW}.

Let $(M,Z)$ be a smooth manifold with an \textcolor{black}{embedded} smooth hypersurface $Z$, which we shall call the \emph{critical set}. The \emph{$b$-tangent bundle} ${}^{b}TM$ is the vector bundle whose sections are vector fields tangent to $Z$. We call these sections \textit{$b$-vector fields}, and note that they are stable under the Lie bracket of vector fields, making them a Lie subalgebra of $Vect(M)$. We obtain the $b$-cotangent bundle ${}^bT^*M$ by taking the dual of ${}^{b}TM$, which allows us to define $b$-forms of degree $k$ as sections $\omega\in\Gamma(\bigwedge^k({}^bT^*M)):={}^b\Omega^k(M)$. Since the space of $b$-vector fields is involutive, we can define the exterior derivative $d:{}^b\Omega^k(M)\rightarrow{}^b\Omega^{k+1}(M)$ of $b$-forms in the natural way, as in \cite{MS21}:
\begin{align*}
    d\omega(V_0,\dots,V_k) = \sum_i (-1)^i &V_i(\omega(V_0,\dots,\hat{V_i},\dots,V_k)) + \\ 
    &+ \sum_{i<j}(-1)^{i+j}\omega([V_i,V_j],V_0,\dots,\hat{V_i},\dots,\hat{V_j},\dots,V_k).
\end{align*}

Given a defining function $z$ of the critical set $Z=\{z=0\}$, then $b$-forms admit a simple decomposition given by the following lemma, which makes evident the usefulness of this framework when dealing with differential forms with logarithmic singularities.

\begin{lemma}[\cite{GMP14}]\label{lemma:decomposition}
Let $\omega\in{}^b\Omega^k(M)$ be a $b$-form of degree $k$. Then $\omega$ decomposes as follows:
    \begin{equation*}
        \omega = \alpha \wedge \frac{dz}{z}+ \beta, \quad \alpha\in\Omega^{k-1}(M),\; \beta\in\Omega^k(M).
    \end{equation*} 
\end{lemma}

This decomposition is only unique near the critical surface $Z$. The exterior derivative for differential $b$-forms defined above is then equivalent to
\begin{equation*}
    d\omega :=d\alpha\wedge \frac{dz}{z} + d\beta.
\end{equation*}

\begin{remark}
    This definition  agrees with the usual exterior $d$ operator on $M\setminus Z$ and also extends smoothly over $M$ as a section of $\bigwedge^{k+1}({}^b T^* M)$. Furthermore, just as with the usual exterior derivative, it gives rise to a chain complex along with its $b$-cohomology \cite{MS21}.
\end{remark}

With these constructions, one can translate any definitions for usual geometric structures to their ``$b$ counterpart". For example, we can define a $b$-contact form and its $b$-Reeb vector field as follows.

\begin{definition}\label{def:b-structures}

        A $b$-contact form is a $b$-form of degree one on an odd-dimensional $b$-manifold $\alpha\in{}^b\Omega^1(M^{2n+1})$ such that $\alpha\wedge (d\alpha)^n$ is non-vanishing as a section of $\bigwedge^{2n+1}({}^bT^*M)$\textcolor{black}{, meaning that this defines a $b$-volume form}. Its kernel $\ker \alpha\subset {}^bTM$ is called a $b$-\emph{contact structure} and the associated $b$\emph{-Reeb vector field} is the unique $b$-vector field $R$ such that
        \begin{equation*}
            \begin{cases}
                \iota_R d\alpha = 0 \\
                \iota_R \alpha = 1.
            \end{cases}
        \end{equation*}
\end{definition}

\begin{remark}
    The equations $\iota_R\alpha=1$ and $\iota_R d\alpha=0$ are to be understood in the context of $b$-vector fields and $b$-forms. It is important to note that a $b$-Reeb vector field can vanish as a section of $TM$. This is because in a chart near the critical set $Z=\{z=0\}$, the $b$-vector field $z\frac{\partial}{\partial z}$ is a vanishing section of $TM$, while \textit{non}-vanishing section of ${}^bTM$. This is a fundamental difference between smooth Reeb dynamics and $b$-Reeb dynamics which ultimately leads to a reformulation of the Weinstein conjecture for $b$-contact forms (see Section \ref{sec:singularW}).

    \textcolor{black}{Note also that the $b$-volume form $\alpha\wedge (d\alpha)^n$ is given in a tubular neighbourhood around the critical set by $\frac{dz}{z}\wedge \Omega$, where $i^*\Omega$ is a volume form on $Z$, and $i:Z \xhookrightarrow{} M$ is the inclusion of $Z$ in $M$.}
\end{remark}

The key advantage of this framework is that it allows for a systematic treatment of geometric structures with singularities of the type seen in Lemma \ref{lemma:decomposition}, including an extension of action-angle coordinates and a KAM theorem for $b$-symplectic forms \cite{KMS16}. 

An analogous procedure can be followed to define $b^m$-forms starting from the $b^m$\emph{-tangent bundle}, whose sections are vector fields tangent to $Z$ with order $m$. This requires the presence of an $(m-1)$-jet of $Z$, but the rest of the constructions remain essentially the same (see \cite{S16} for details). The change to McGehee coordinates in the restricted circular $3$-body problem gives rise to a $b^3$-symplectic form.

\begin{example}[\cite{KMS16, DKRS}]\label{ex:McGehee}
    Consider the motion of a massless object in the gravitational field of two bodies $q_1,q_2$ with masses $(1-\mu)$ and $\mu$ respectively, orbiting in circular Keplerian motion. The corresponding Hamiltonian is
    \begin{equation*}
        H(q,p,t) =  K+U(t) = \frac{|p|^2}{2} - \frac{1-\mu}{|q - q_1(t)|} - \frac{\mu}{|q - q_2(t)|},
    \end{equation*}
    where $(q,p)$ are the positions and momenta of the massless object. Passing to rotating coordinates we can fix the positions of the massive bodies at $q_1 = (\mu,0)$ and $q_2 = (-(1-\mu),0)$ to eliminate the dependence on time at the cost of adding a term to the potential,
    \begin{equation*}
        H(q,p) = \frac{|p|^2}{2} - \frac{1-\mu}{|q - q_1|} - \frac{\mu}{|q - q_2|} + p_1q_2-p_2q_1.
    \end{equation*}
    After a symplectic change to polar coordinates $(r,\alpha, P_r, P_\alpha)$, the symplectic form becomes $\omega = dr\wedge dP_r + d\alpha\wedge dP_\alpha$, and the Hamiltonian is expressed as
    \begin{equation*}
        H(r,\alpha,P_r,P_\alpha) = \frac{P_r^2}{2} + \frac{P_\alpha^2}{2r^2} - P_\alpha + U(r,\alpha).
    \end{equation*}
    Now, to study the behaviour of the system near infinity, it is convenient to apply the McGehee change of coordinates $r = \frac{2}{x^2}$ (see for instance \cite{DKRS}). If we do not require this change to be symplectic, we obtain the Hamiltonian
    \begin{equation*}
        \frac{P_r^2}{2} + \frac{x^4P_\alpha^2}{8} - P_\alpha - x^4\frac{1-\mu}{4 - 4\mu x^2\cos\alpha + \mu^2x^4} - x^4\frac{\mu}{4 - 4(1-\mu)x^2\cos\alpha + (1-\mu)^2x^4},
    \end{equation*}
    and the symplectic form becomes the $b^3$-symplectic form $\omega = -4\frac{dx}{x^3}\wedge P_r + d\alpha \wedge dP_\alpha$. Using this expression, it is shown in \cite{MO21} that positive energy level sets are $b^3$-contact manifolds and that the Hamiltonian vector field is $b^3$-Reeb. Furthermore, the $b^3$-Reeb field has infinitely many periodic orbits on the manifold at infinity $\{x=0\}$.
    
\end{example}

%%%%%%%%%%%%%%%%%%%%%%%%%%%%%%%%%%%%

\subsection{The singular Weinstein conjecture}\label{sec:singularW}

In view of this large class of dynamics generated by $b$-Reeb vector fields, closely related to the dynamics generated by smooth Reeb vector fields, a series of interesting questions emerge. How does the $b$-Reeb vector field behave on the critical surface? How does adding a critical surface and introducing a singularity in the contact form along this surface affect the dynamics away from the critical surface? What can we say about Reeb dynamics on open manifolds when compactifying the manifold yields $b$-Reeb dynamics? Perhaps the most interesting question is, does the Weinstein conjecture also hold on manifolds with boundary or more generally on compact $b$-manifolds, and if not, how would it translate to this setting?

A detailed analysis of $b$-contact forms and their dynamics can be found in \cite{MO18}, where the first question is answered for $b$-manifolds of dimension 3. Indeed, the fact that the $b$-Reeb vector field of Example \ref{ex:McGehee} has infinitely many periodic orbits on the critical set is not a coincidence.

\begin{proposition}[\cite{MO18}]\label{prop:reebatz}
    Let $(M,Z,\alpha)$ be a $b$-contact manifold of dimension $3$, and write $\alpha = f\frac{dz}{z} + \beta$ with $f\in C^\infty(M)$ and $\beta\in\Omega^1(M)$ as in Lemma \ref{lemma:decomposition}. Then the restriction on $Z$ of the $2$-form $\omega = fd\beta + \beta\wedge df$ is symplectic and the $b$-Reeb vector field $R_\alpha$ is Hamiltonian on $Z$ with respect to $\omega$ with Hamiltonian function $-f|_Z$, i.e. $\iota_{R_\alpha}\omega = df$. The Hamiltonian $-f|_Z$ is called the \emph{exceptional Hamiltonian} associated with $\alpha$.
\end{proposition}

As is observed in \cite{MO21}, when $Z$ is closed the exceptional Hamiltonian $-f|_Z$ cannot be locally constant, because then the symplectic form would also be exact ($\omega = d(f\beta)$). By Stokes' Theorem one sees that $\omega$ cannot be exact and an area form at the same time. This implies that there are infinitely many periodic orbits on closed critical sets of 3-dimensional $b$-manifolds. The dynamics of $b$-Reeb vector fields are further studied in \cite{MO21}, where the authors prove that there exist compact $b$-contact manifolds in any dimension whose $b$-Reeb vector fields have no periodic orbits away from the critical set. However, all of these examples exhibit \textit{singular periodic orbits}. 

\begin{definition}\label{def:spo}
    Let $(M,Z,\alpha)$ be a $b$-contact manifold. A \textit{singular periodic orbit} is an integral curve $\gamma:\mathbb{R}\rightarrow M\setminus Z$ of the $b$-Reeb vector field such that $\lim_{t\rightarrow \pm\infty} \gamma(t) = p_\pm\in Z$. An \textit{escape orbit} is an integral curve such that at least one of the semiorbits has a stationary limit point on $Z$ (see Figure \ref{fig:examplesingularperiodic}).
\end{definition}

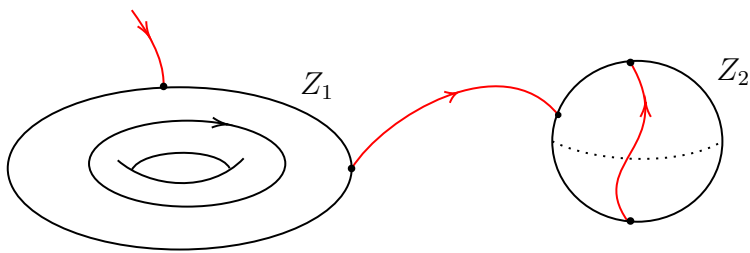
\begin{figure}[H]
    \centering
    \tikzset{every picture/.style={line width=0.75pt}} %set default line width to 0.75pt        

\begin{tikzpicture}[x=0.75pt,y=0.75pt,yscale=-1,xscale=1]

%Curve Lines [id:da6450894908163725] 
\draw [red]   (414.06,108.41) .. controls (439.19,150.83) and (391.24,158.24) .. (412.61,188.42) ;
\draw  [red] (419.04,135.9) .. controls (420.5,133.34) and (421.37,130.79) .. (421.65,128.23) .. controls (421.97,130.79) and (422.87,133.32) .. (424.38,135.86) ;
%Curve Lines [id:da0320077196690598] 
\draw [red]    (273.47,161.92) .. controls (289.73,135.73) and (349.45,100.1) .. (377.64,134.88) ;
\draw  [red] (320.13,123.14) .. controls (322.59,123.97) and (324.74,124.17) .. (326.58,123.74) .. controls (325.06,124.79) and (323.86,126.47) .. (322.99,128.78) ;
%Curve Lines [id:da9389739320938872] 
\draw  [red]  (178.61,119.76) .. controls (178.71,105) and (169.9,92.39) .. (162.55,82) ;
\draw  [red] (169.86,87.6) .. controls (169.94,90.54) and (170.5,93.2) .. (171.55,95.56) .. controls (170.02,93.48) and (168,91.7) .. (165.49,90.2) ;
%Straight Lines [id:da3467971493729016] 
%\draw    (310,222.17) -- (254.88,202.51) ;
%\draw [shift={(253,201.83)}, rotate = 19.63] [color={rgb, 255:red, 0; green, 0; blue, 0 }  ][line width=0.75]    (10.93,-3.29) .. controls (6.95,-1.4) and (3.31,-0.3) .. (0,0) .. controls (3.31,0.3) and (6.95,1.4) .. (10.93,3.29)   ;
%Straight Lines [id:da09208308132394216] 
%\draw    (335,222.17) -- (385.67,188.6) ;
%\draw [shift={(387.33,187.5)}, rotate = 146.48] [color={rgb, 255:red, 0; green, 0; blue, 0 }  ][line width=0.75]    (10.93,-3.29) .. controls (6.95,-1.4) and (3.31,-0.3) .. (0,0) .. controls (3.31,0.3) and (6.95,1.4) .. (10.93,3.29)   ;
%Shape: Ellipse [id:dp7821200803445583] 
\draw [black] (100,161.87) .. controls (100,139.25) and (138.83,120.92) .. (186.73,120.92) .. controls (234.63,120.92) and (273.47,139.25) .. (273.47,161.87) .. controls (273.47,184.5) and (234.63,202.83) .. (186.73,202.83) .. controls (138.83,202.83) and (100,184.5) .. (100,161.87) -- cycle ;
%Curve Lines [id:da5552988530490628] 
\draw [black]   (155.47,157.79) .. controls (171.61,171.8) and (201.86,174.8) .. (219.01,156.79) ;
%Curve Lines [id:da886369424392208] 
\draw [black] (162.53,162.79) .. controls (167.57,151.78) and (205.9,150.78) .. (211.95,161.79) ;

%Shape: Ellipse [id:dp58823925084254] 
\draw   (374.63,148.2) .. controls (374.63,125.78) and (393.84,107.6) .. (417.54,107.6) .. controls (441.23,107.6) and (460.44,125.78) .. (460.44,148.2) .. controls (460.44,170.62) and (441.23,188.8) .. (417.54,188.8) .. controls (393.84,188.8) and (374.63,170.62) .. (374.63,148.2) -- cycle ;
%Curve Lines [id:da3889744751268014] 
\draw [dash pattern={on 0.84pt off 2.51pt}]  (374.63,148.2) .. controls (396.08,158.95) and (438.99,161.13) .. (460.44,148.2) ;

%Shape: Ellipse [id:dp6132049913503561] 
\draw  [fill={rgb, 255:red, 0; green, 0; blue, 0 }  ,fill opacity=1 ] (412.61,188.42) .. controls (412.61,187.65) and (413.26,187.01) .. (414.06,187.01) .. controls (414.87,187.01) and (415.52,187.65) .. (415.52,188.42) .. controls (415.52,189.2) and (414.87,189.83) .. (414.06,189.83) .. controls (413.26,189.83) and (412.61,189.2) .. (412.61,188.42) -- cycle ;
%Shape: Ellipse [id:dp06181907191784286] 
\draw  [fill={rgb, 255:red, 0; green, 0; blue, 0 }  ,fill opacity=1 ] (412.61,108.41) .. controls (412.61,107.63) and (413.26,107) .. (414.06,107) .. controls (414.87,107) and (415.52,107.63) .. (415.52,108.41) .. controls (415.52,109.19) and (414.87,109.82) .. (414.06,109.82) .. controls (413.26,109.82) and (412.61,109.19) .. (412.61,108.41) -- cycle ;
%Shape: Ellipse [id:dp8226942121731127] 
\draw  [fill={rgb, 255:red, 0; green, 0; blue, 0 }  ,fill opacity=1 ] (376.5,134.88) .. controls (376.5,134.25) and (377.01,133.74) .. (377.64,133.74) .. controls (378.27,133.74) and (378.78,134.25) .. (378.78,134.88) .. controls (378.78,135.51) and (378.27,136.02) .. (377.64,136.02) .. controls (377.01,136.02) and (376.5,135.51) .. (376.5,134.88) -- cycle ;
%Shape: Ellipse [id:dp9487342858897618] 
\draw  [fill={rgb, 255:red, 0; green, 0; blue, 0 }  ,fill opacity=1 ] (272,161.92) .. controls (272,161.14) and (272.62,160.51) .. (273.39,160.51) .. controls (274.17,160.51) and (274.79,161.14) .. (274.79,161.92) .. controls (274.79,162.7) and (274.17,163.33) .. (273.39,163.33) .. controls (272.62,163.33) and (272,162.7) .. (272,161.92) -- cycle ;
%Shape: Ellipse [id:dp9231647014991853] 
\draw  [fill={rgb, 255:red, 0; green, 0; blue, 0 }  ,fill opacity=1 ] (177.22,120.5) .. controls (177.22,119.72) and (177.84,119.09) .. (178.61,119.09) .. controls (179.38,119.09) and (180.01,119.72) .. (180.01,120.5) .. controls (180.01,121.28) and (179.38,121.91) .. (178.61,121.91) .. controls (177.84,121.91) and (177.22,121.28) .. (177.22,120.5) -- cycle ;
%Shape: Ellipse [id:dp560243196582926] 
\draw  [black ] (141,159.48) .. controls (141,147.5) and (163.16,137.79) .. (190.5,137.79) .. controls (217.84,137.79) and (240,147.5) .. (240,159.48) .. controls (240,171.46) and (217.84,181.17) .. (190.5,181.17) .. controls (163.16,181.17) and (141,171.46) .. (141,159.48) -- cycle ;
\draw  [black ] (203.48,136.23) .. controls (205.87,137.95) and (208.32,139.11) .. (210.84,139.7) .. controls (208.26,139.67) and (205.61,140.19) .. (202.9,141.28) ;

% Text Node
%\draw (318.06,219.63) node [anchor=north west][inner sep=0.75pt]   [align=left] {$\displaystyle Z$};
\draw (256,120) node {$\displaystyle Z_1$};
\draw (466,113) node {$\displaystyle Z_2$};

\end{tikzpicture}
    \caption{Examples of an escape orbit (on the left, tending to a point on the critical torus) and two singular periodic orbits. The critical set is a disjoint union of a torus $Z_1$ and a sphere $Z_2$, and a $b$-Reeb orbit on the critical torus is depicted in black.}
    \label{fig:examplesingularperiodic}
\end{figure}

Since there exist $b$-Reeb vector fields with no periodic orbits away from $Z$, in \cite{MO21} the authors conjecture that these singular periodic orbits are the appropriate invariant dynamical sets to consider to reformulate the Weinstein conjecture for $b$-contact manifolds.

\begin{conjecture}[Singular Weinstein conjecture]\label{conj:singularweinstein}
Let $(M, Z, \alpha)$ be a compact $b^m$-contact manifold. Then there exists at least one singular periodic orbit.
\end{conjecture}

The exceptional Hamiltonian introduced in Proposition \ref{prop:reebatz} also plays an important role in \cite{MOP22}, where the authors use it to prove partial results towards the singular Weinstein conjecture. In particular, an analysis of the exceptional Hamiltonian of a certain class of $b$-contact forms which we call \textit{Melrose $b$-contact forms} (see Definition \ref{def:melrosebcontact}) is used to prove lower bounds on the number of escape orbits in terms of the topology of the critical surface. Since singular periodic orbits are a special case of escape orbits, this is a first step towards the singular Weinstein conjecture. Our aim in this article is to give a lower bound on the number of these orbits. Interestingly enough, the lower bounds obtained resemble the ``$2$ or $\infty$-conjecture", in the spirit of the recent result by Colin--Dehornoy--Rechtman in \cite{CDR22}.

In view of Example \ref{ex:McGehee}, we note that if the critical surface arises from a compactification procedure on the phase space of a dynamical system, these results provide lower bounds on the number of escape orbits in the usual sense, that is, orbits going ``to infinity''. 

To better understand the exceptional Hamiltonian it is convenient to look at how it arises in the context of $b$-Beltrami fields, which we introduce in the following subsection.

%%%%%%%%%%%%%%%%%%

\subsection{$b$-Beltrami vector fields}\label{sec:bBeltrami}

Following the ideas of Sullivan in \cite{Sullivan}, Etnyre and Ghrist established a connection between contact geometry and a special class of stationary solutions to the incompressible Euler equations of hydrodynamics, known as \emph{Beltrami vector fields}  \cite{EG}. Since then, there has been a rich interaction between the two fields \cite{PR, CMP21}. Beltrami vector fields are eigenfunctions of the curl operator with respect to a metric and a distinguished volume form on a Riemannian $3$-manifold. As is done in \cite{CMP19}, we can extend this notion to $b$-manifolds with a Riemannian $b$-metric, which is a bilinear\textcolor{black}{, symmetric} positive-definite section $g\in\Gamma({}^{b}T^*M\otimes{}^{b}T^*M)$, and a $b$-volume form (see Definition \ref{def:b-structures}).

\begin{definition}
    A $b$\emph{-Beltrami vector field} $X$ is a vector field on a Riemannian $b$-manifold with a distinguished $b$-volume form such that $\curl X = \lambda X$, for some nonzero constant $\lambda$, where the curl operator is defined as usual by $\iota_{\curl X}\mu =d(\iota_Xg)$, with respect to the $b$-metric $g$ and the $b$-volume form $\mu$.
\end{definition}

We emphasize that the volume form and the metric can be chosen independently one of each other in the previous definition (although it is customary to assume that $\mu$ is the Riemannian volume associated to the metric $g$).

The correspondence between $b$-Beltrami vector fields and $b$-Reeb fields is then the following.

\begin{theorem}[\cite{CMP19}]\label{thm:contactbeltrami}
    Let $(M,Z)$ be $3$-dimensional $b$-manifold. For each $b$-Beltrami vector field $X$ which is non-vanishing as a section of ${}^{b}TM$ there is a $b$-contact form for which $X$ is $b$-Reeb up to rescaling. Conversely, for each $b$-contact form with a rescaling of the $b$-Reeb vector field $X$ there is a $b$-metric and a distinguished $b$-volume form for which $X$ is $b$-Beltrami.
\end{theorem}

\begin{remark}\label{rm:contactfrombeltrami}
    Given a $b$-metric $g$ with a Beltrami field $X$, the way the corresponding $b$-contact form is constructed in Theorem \ref{thm:contactbeltrami} is by contracting $X$ with the metric. In other words, $X$ is $b$-Reeb up to rescaling for the $b$-contact form
    \begin{equation*}
        \alpha = g(X,\cdot).
    \end{equation*}
\end{remark}

In Section \ref{sec:escapeBeltrami} we consider a special type of $b$-metrics, which have a specific description near the critical surface. These $b$-metrics allow a simple description of the corresponding $b$-contact forms and their exceptional Hamiltonians.

\begin{definition}\label{def:asymptoticallyexact}
    An \emph{asymptotically exact $b$-metric} is a $b$-metric $g$ on $(M,Z)$ which can be split in a tubular neighborhood $(P,z):\mathcal{N}(Z)\rightarrow Z\times (-\varepsilon,\varepsilon)$ into
    \begin{equation}\label{eq:splitmetric}
        g= P^*h + \frac{dz^2}{z^2},
    \end{equation}
    where $h$ is a smooth metric on $Z$. The space of asymptotically exact $b$-metrics of class $C^k$ on $(M,Z)$ is denoted by $\mathcal{G}^k_b$. In the neighborhood $\mathcal{N}(Z)$ it inherits the $C^k$-topology of the space of $C^k$ Riemannian metrics on $Z$ via the map $P$. With this topology it makes sense to speak of \textit{generic} metrics within the class $\mathcal{G}^k_b$.
\end{definition}

Asymptotically exact $b$-metrics are a special class of $b$-metrics first introduced by Melrose in \cite{M93}. We accordingly name the corresponding $b$-contact forms \textit{Melrose} $b$\textit{-contact forms}.

\begin{definition}\label{def:melrosebcontact}
    A \emph{Melrose $b$-contact form} is a $b$-contact form which is constructed from an asymptotically exact $b$-metric and a non-vanishing (as a section of ${}^{b}TM$) $b$-Beltrami vector field via the correspondence of Theorem \ref{thm:contactbeltrami} and Remark \ref{rm:contactfrombeltrami}.
\end{definition}

%An interesting property of Melrose $b$-contact forms is that, \textcolor{blue}{generically}, the exceptional Hamiltonian can be read directly from the $b$-Beltrami vector field. Before stating precisely the definition of \emph{generic} $b$-Beltrami vector field, we state the following lemma.
\begin{remark}\label{rem:bBeltramiexceptional}
    Because of the correspondence of Theorem \ref{thm:contactbeltrami}, we talk about Melrose $b$-contact forms and asymptotically exact $b$-metrics interchangeably when the $b$-Beltrami field is non-vanishing. It also makes sense to talk about the exceptional Hamiltonian (see Proposition \ref{prop:reebatz}) of a $b$-Beltrami vector field in this case. 
\end{remark} 

    An interesting property of Melrose $b$-contact forms is that, in coordinate charts for which the corresponding asymptotically exact $b$-metric splits as in Definition \ref{def:asymptoticallyexact}, the exceptional Hamiltonian is a component of the  $b$-Beltrami vector field. 

\begin{proposition}[\cite{MOP22}]
         Let $X$ be a $b$-Beltrami vector field for an asymptotically exact $b$-metric $g$ and let $(x,y,z)$ be coordinates near $Z$ for which the metric splits as in Definition \ref{def:asymptoticallyexact}. Then the corresponding exceptional Hamiltonian is $-X_z|_Z$, where
        \begin{equation*}
        X = X_x\frac{\partial}{\partial x} + X_y\frac{\partial}{\partial y}  + zX_z\frac{\partial}{\partial z}.
    \end{equation*}
\end{proposition}
    Viewing the exceptional Hamiltonian as a component of an eigenfunction of the curl operator makes it possible to apply powerful results from spectral theory. Indeed, in \cite{MOP22} the authors prove that the exceptional Hamiltonian of a $b$-Beltrami field is an eigenfunction of the Laplacian $\Delta_h$ with respect to the metric $h$ on $Z$, where $h$ is the component of the asymptotically exact $b$-metric as in Definition \ref{def:asymptoticallyexact}. Thus, they obtain important properties of generic exceptional Hamiltonians using a theorem by Uhlenbeck in \cite{U76}. The summarized result is the following, and the meaning of \textit{generic exceptional Hamiltonian} is explained in the remark after it.

\begin{theorem}[\cite{MOP22}]\label{lemma:splitting}\label{thm:hammorse}
    There exists a residual set $\widehat{\mathcal{G}^k_b}\subset\mathcal{G}^k_b$ in the space of asymptotically exact $b$-metrics such that if $g\in \widehat{\mathcal{G}^k_b}$, then in local coordinates as in Definition \ref{def:asymptoticallyexact}, the function $-X_z|_Z$ is a Morse function on $Z$, and zero is a regular value, provided that it is not locally constant.
\end{theorem}

\begin{remark} \label{rem:residual}
    Here, \textit{residual} means that $\widehat{\mathcal{G}^k_b}$ is a countable intersection of open and dense subsets of $\mathcal{G}^k_b$, as in Uhlenbeck's Theorem \cite{U76}. In particular, $\widehat{\mathcal{G}^k_b}$ is dense in $\mathcal{G}^k_b$. When we say \textit{generic exceptional Hamiltonian}, \textit{generic $b$-Beltrami vector field} and \textit{generic Melrose $b$-contact form}, we mean that the underlying metric $g$ is in the residual set $\widehat{\mathcal{G}^k_b}$.
\end{remark}

\begin{remark}\label{rem:derivatives}
    In the proof of Theorem \ref{thm:hammorse}, a computation also yields an expression for the components $X_x$ and $X_y$ on $Z$ of a $b$-Beltrami field $X$ for an asymptotically exact $b$-metric in terms of $-X_z|_Z$  (see \cite{MOP22}, Proposition 3.3):
    \begin{equation} \label{eq:beltramiexpression}
        \begin{cases}
            X_x|_Z = \frac{-1}{\lambda\sqrt{\det h}}\frac{\partial (-X_z|_Z)}{\partial y}\\
            X_y|_Z = \frac{1}{\lambda\sqrt{\det h}}\frac{\partial (-X_z|_Z)}{\partial x},
        \end{cases}
    \end{equation}
    where $\lambda$ is the eigenvalue of $X$, that is $\curl X = \lambda X$. In fact, Theorem \ref{thm:hammorse} does not require that the $b$-Beltrami field is non-vanishing as a section of ${}^bTM$.
    Nevertheless, we observe that $X|_Z$ is a Hamiltonian vector field whose Hamiltonian function is precisely $-X_z|_Z$, and the symplectic form on $Z$ is simply $\lambda$ multiplied by the Riemannian area form associated to the metric $h$.
    %since there is no $b$-contact form associated to $X$ and thus Proposition \ref{prop:reebatz} does not yield a symplectic form on $Z$ for which $X$ is Hamiltonian.  Note that nevertheless the critical points of $X$ on $Z$ are exactly the critical points of $-X_z|Z$.
\end{remark}

%Theorem \ref{thm:hammorse}  turns out to be very useful in studying the dynamics generated by generic Melrose $b$-contact forms. Indeed, it was proved that a generic $b$-Beltrami vector field, which is not identically zero, admits at least $2+b_1(Z)$ escape orbits, where $b_1(Z)$ is the first Betti number of the critical set $Z$. This was achieved by studying the linear stability around the critical points of the associated exceptional Hamiltonian. Using the contact-Beltrami correspondence of Theorem \ref{thm:contactbeltrami}, this result can then be viewed as follows. Since the exceptional Hamiltonian of a $b$-Reeb vector field on a closed critical set cannot be locally constant (see Proposition \ref{prop:reebatz} and the comment right after it), the $b$-Reeb vector field is not identically zero on $Z$,  and the corresponding result for $b$-Reeb dynamics of generic Melrose $b$-contact forms yields the existence of $2+b_1(Z)$ escape orbits.

%As we will see in the next section, the linear stability study around critical points of $b$-Reeb vector fields can be performed by more elementary means, without making use of an auxiliary  asymptotically exact $b$-metric, nor Uhlenbeck's result or the correspondence with $b$-Beltrami fields.

\textcolor{black}{Theorem \ref{thm:hammorse} is used in \cite{MOP22} to study the dynamics generated by generic Melrose $b$-contact forms through an analysis of the linear stability of the critical points of the associated exceptional Hamiltonian. However, as we will see in the next section, the linear stability study around critical points of $b$-Reeb vector fields can be performed by more elementary means, without making use of an auxiliary  asymptotically exact $b$-metric, Uhlenbeck's result or the correspondence with $b$-Beltrami fields.}

%%%%%%%%%%%%%%%%%%%%%%%%%%%%%%%%%%%%%

\section{At least $2N$ or infinitely many escape orbits}\label{sec:mainresult}

In this section we prove our first main theorem. It is a stronger version of \cite[Theorem 1.3]{MOP22}, and it follows from a detailed analysis of the critical points of a $b$-Reeb vector field on the critical set $Z$ and elementary Morse theory. In contrast with~\cite[Theorem 1.3]{MOP22} we do not refer to any compatible asymptotically exact $b$-metric. 

\begin{theorem}\label{thm:main}
	Let $\alpha$ be a $b$-contact form on a $3$-dimensional manifold $(M,Z)$ without boundary, with $Z$ a closed embedded surface in $M$.
    Then there exists a $b$-contact form $C^\infty$-close to $\alpha$, such that the associated $b$-Reeb vector field has either
	\begin{enumerate}[(i)]
		\item infinitely many escape orbits if $b_1(Z)>0$, or
		\item at least $2N$ escape orbits if $b_1(Z) = 0$, where $N$ is the number of connected components of $Z$. 
	\end{enumerate}
 Moreover, the set of $b$-contact forms exhibiting these properties is open in the $C^\infty$-topology.
\end{theorem}
Here $b_1(Z)$ denotes the first Betti number of the critical surface $Z$. We discuss what happens when $M$ has boundary in Remark \ref{rem:boundary} below.
\begin{remark}
%HABRIA QUE PONER BREVEMENTE LO QUE QUEREMOS DECIR POR LA TOPOLOGIA EN EL CONTEXTO B, DADO QUE LAS FORMAS SON SINGULARES Y NO ESTAN ACOTADAS CERCA DE LA SUPERFICIE CRITICA. \textcolor{blue}{Vale, he puesto en la siguiente remark.}
\textcolor{black}{The $C^\infty$-topology on the space of $b$-forms is defined as follows. Away from a tubular neighbourhood of $Z$, the topology coincides with the usual one, as $b$-forms are smooth differential forms away from $Z$. Around the critical set, given a $b$-form $\omega$ and its decomposition as in Lemma \ref{lemma:decomposition}, that is $\omega=\alpha \wedge \frac{dz}{z}+\beta$, we define a smooth form in the tubular neighbourhood by $\overline{\omega}:=\alpha \wedge dz+\beta$. We say that in the tubular neighbourhood two $b$-forms $\omega_1, \omega_2$ are $C^\infty$-close if $\overline{\omega_1}, \overline{\omega_2}$ are $C^\infty$-close in the sense of smooth differential forms defined on this neighbourhood.}
\end{remark}

\begin{proof}[Proof of Theorem \ref{thm:main}]
    The proof follows essentially the proof of \cite[Theorem 1.3]{MOP22}, without making use of the contact-Beltrami correspondence nor spectral geometry techniques.

    As before, a $b$-contact form in a tubular neighbourhood around the critical set $Z$, denoted by $\mathcal{N}(Z)$, is given by
    $$\alpha=f\frac{dz}{z}+\beta,$$
    where $f\in C^\infty(\mathcal{N}(Z))$ and $\beta\in \Omega^1(\mathcal{N}(Z))$.

    By a $C^\infty$-small perturbation, we can assume that the function $f$ in the above decomposition for  $\alpha$ restricts to a Morse function on $Z$. Indeed, we choose a $b$-contact form that is $C^\infty$-close to $\alpha$ as
    \begin{equation}
        \widetilde{\alpha}:=(f+\epsilon h) \frac{dz}{z}+\beta,
    \end{equation}

    where $h$ is a $C^\infty$-small function, which is supported in a tubular neighbourhood of the critical set, such that $(f+\epsilon h)|_Z$ is a Morse function. It is clear that this is still a $b$-contact form (as this is an open condition) if $\epsilon$ is small enough. The reason for perturbing $f$ to become Morse will come apparent from the analysis carried out around a tubular neighbourhood of the associated $b$-Reeb vector field. We will thus assume from now on that $\alpha$ satisfies this condition.
    
    The associated $b$-Reeb vector field in this tubular neighbourhood is given by
    $$R=gz\frac{\partial}{\partial z}+Y,$$
    where $g\in C^\infty(\mathcal{N}(Z))$ and $Y\in \mathfrak{X}(\mathcal{N}(Z))$ such that $\iota_Y(dz)=0$.

    By Proposition \ref{prop:reebatz}, the restriction of the smooth $2$-form 
    $$\omega:=fd\beta+\beta \wedge df$$
    to $Z$ is symplectic. This implies that at a critical point $p\in Z$ of $f|_Z$, $f(p)\neq 0$. In other words, $0$ is a regular value of $f|_Z$.

    By the same proposition, $R|_Z$ is a Hamiltonian vector field with respect to $\omega|_Z$, and the exceptional Hamiltonian is given by $H:=-f|_Z$. We denote this Hamiltonian vector field by
    $$\overline{R}:=R|_Z=Y|_Z.$$

    It follows that at a critical point $p\in Z$ of  $H$ we have a zero of the $b$-Reeb vector field, that is, $R_p=\overline{R}_p=0$. Thus, by the assumption that $Z$ is closed, $R$ admits at least two zeroes (corresponding to the maximum and minimum values of $H$). Furthermore, at a critical point $p\in Z$, we have $g(p)\neq 0$ because the Reeb condition $\alpha(R)=1$ yields
$$1=\alpha(R)|_p=f(p)g(p)+\beta_p(Y_p)=f(p)g(p).$$
    We now study the linear stability of $R$ around the critical points. At a critical point $p$, the differential of $R$ is given by
$$DR(p)=\left.\begin{pmatrix} D\overline{R} & * \\ 0 & g \end{pmatrix}\right|_p.$$
    We choose a Darboux chart around the critical point in $Z$, so that in local coordinates with $\omega=-dx\wedge dy$ we obtain
    \begin{equation*}
    DR(p)=\left.\begin{pmatrix} H_{xy} & H_{yy}& *\\-H_{xx} &-H_{xy} & * \\ 0 & 0& g \end{pmatrix}\right| _p.
  \end{equation*}
    It is now easy to determine the linear stability at $p$ by looking at the eigenvalues of this matrix. The eigenvalues are $\lambda_+,\lambda_-$ and $\lambda_z$, where $\lambda_+$ and $\lambda_-$ are eigenvalues of the first $2\times 2$ minor,
    \begin{equation*}
        \lambda_\pm = \pm\sqrt{-\Hess H(p)},
    \end{equation*}
    and $\lambda_z=g(p)\neq 0$. Notice that $\lambda_\pm\neq 0$ because we assume that $f|_Z$ (and hence $H$) is a Morse function.
    
    There are two situations to consider, according to the sign of $\Hess H(p)$:

\begin{itemize}
    
    \item $\Hess H(p) < 0$: In this case, the critical point of $R$ is hyperbolic and there is a two-dimensional stable or unstable (depending on the sign of $g(p)$) manifold at $p$ that is transverse to $Z$.
    
    \item $\Hess H(p) > 0$:  In this case, the critical point of $R$ is non-hyperbolic and there is a one-dimensional stable or unstable (depending on the sign of $g(p)$) manifold at $p$ that is transverse to $Z$, the center manifold being $Z$.

\end{itemize}

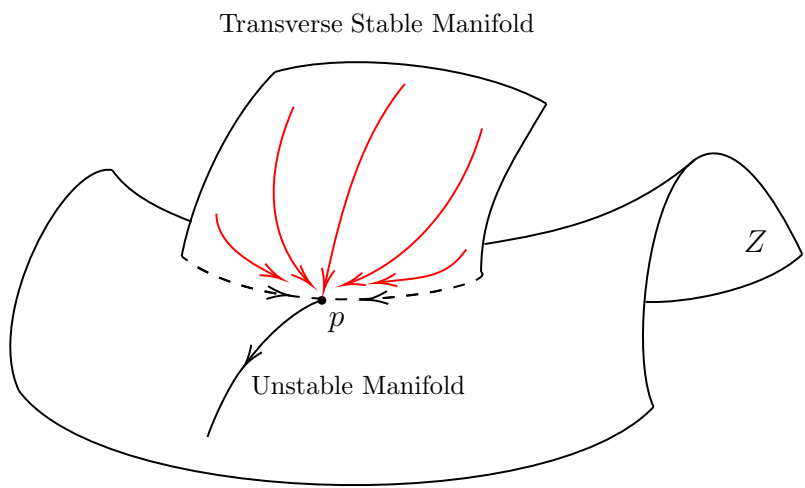
\begin{figure}[H]
	\centering
 
\tikzset{every picture/.style={line width=0.75pt}} %set default line width to 0.75pt        

\begin{tikzpicture}[x=0.75pt,y=0.75pt,yscale=-1,xscale=1]

%Curve Lines [id:da7506645270639489] 
\draw    (124,194.5) .. controls (168,253) and (390,259.5) .. (444,202.5) ;
%Curve Lines [id:da849771550754947] 
\draw    (444,202.5) .. controls (426,164.5) and (454,-10.5) .. (519,125.5) ;
%Curve Lines [id:da6879346367250276] 
\draw    (440,149.5) .. controls (459,150.5) and (500,143.5) .. (519,125.5) ;
%Curve Lines [id:da592752250266078] 
\draw    (124,194.5) .. controls (106,156.5) and (147,78.5) .. (171,83) ;
%Shape: Circle [id:dp41847273689429154] 
\draw  [fill={rgb, 255:red, 0; green, 0; blue, 0 }  ,fill opacity=1 ] (275,148.71) .. controls (275,147.77) and (275.77,147) .. (276.71,147) .. controls (277.66,147) and (278.43,147.77) .. (278.43,148.71) .. controls (278.43,149.66) and (277.66,150.43) .. (276.71,150.43) .. controls (275.77,150.43) and (275,149.66) .. (275,148.71) -- cycle ;
%Curve Lines [id:da259874297483764] 
\draw    (206,126.5) .. controls (211,99.5) and (228,58.5) .. (253,33.5) ;
%Curve Lines [id:da08871153721675484] 
\draw    (253,33.5) .. controls (290,22.5) and (358,30.5) .. (390,49.5) ;
%Curve Lines [id:da18603691559118185] 
\draw    (357,134.5) .. controls (357,101.5) and (371,81.5) .. (390,49.5) ;
%Curve Lines [id:da1519628280902212] 
\draw  [dash pattern={on 4.5pt off 4.5pt}]  (357,134.5) .. controls (368,141.5) and (257,165.5) .. (206,126.5) ;
%Curve Lines [id:da08341638516075256] 
\draw [color={rgb, 255:red, 255; green, 0; blue, 0 }  ,draw opacity=1 ]   (318.5,39.5) .. controls (296.83,66.1) and (286.96,99.64) .. (278.07,141.74) ;
\draw [shift={(277.67,143.67)}, rotate = 281.82] [color={rgb, 255:red, 255; green, 0; blue, 0 }  ,draw opacity=1 ][line width=0.75]    (10.93,-3.29) .. controls (6.95,-1.4) and (3.31,-0.3) .. (0,0) .. controls (3.31,0.3) and (6.95,1.4) .. (10.93,3.29)   ;
%Curve Lines [id:da15283310851141918] 
\draw  [dash pattern={on 4.5pt off 4.5pt}]  (214,132.17) .. controls (223.05,137.02) and (238.07,143.12) .. (258.12,145.92) ;
\draw [shift={(260,146.17)}, rotate = 187.24] [color={rgb, 255:red, 0; green, 0; blue, 0 }  ][line width=0.75]    (10.93,-3.29) .. controls (6.95,-1.4) and (3.31,-0.3) .. (0,0) .. controls (3.31,0.3) and (6.95,1.4) .. (10.93,3.29)   ;
%Curve Lines [id:da49425703118574815] 
\draw  [dash pattern={on 4.5pt off 4.5pt}]  (337,143.5) .. controls (330.28,145.42) and (315.57,146.73) .. (300.84,148.3) ;
\draw [shift={(299,148.5)}, rotate = 353.8] [color={rgb, 255:red, 0; green, 0; blue, 0 }  ][line width=0.75]    (10.93,-3.29) .. controls (6.95,-1.4) and (3.31,-0.3) .. (0,0) .. controls (3.31,0.3) and (6.95,1.4) .. (10.93,3.29)   ;
%Curve Lines [id:da1547889957826445] 
\draw    (171,83) .. controls (178.33,93.67) and (191.67,101.67) .. (211,109) ;
%Curve Lines [id:da7459864061272488] 
\draw    (358.73,120.33) .. controls (396.07,114.33) and (430.13,106.97) .. (464.8,77.8) ;
%Curve Lines [id:da7908638362601776] 
\draw    (276.71,148.71) .. controls (268.91,150.66) and (252.35,161.44) .. (238.56,180.52) ;
\draw [shift={(237.5,182)}, rotate = 304.99] [color={rgb, 255:red, 0; green, 0; blue, 0 }  ][line width=0.75]    (10.93,-3.29) .. controls (6.95,-1.4) and (3.31,-0.3) .. (0,0) .. controls (3.31,0.3) and (6.95,1.4) .. (10.93,3.29)   ;
%Curve Lines [id:da13322266780415637] 
\draw [color={rgb, 255:red, 255; green, 0; blue, 0 }  ,draw opacity=1 ]   (262.5,51) .. controls (246.82,86.28) and (249.39,118.68) .. (270.2,139.73) ;
\draw [shift={(271.5,141)}, rotate = 223.67] [color={rgb, 255:red, 255; green, 0; blue, 0 }  ,draw opacity=1 ][line width=0.75]    (10.93,-3.29) .. controls (6.95,-1.4) and (3.31,-0.3) .. (0,0) .. controls (3.31,0.3) and (6.95,1.4) .. (10.93,3.29)   ;
%Curve Lines [id:da27806533064619465] 
\draw [color={rgb, 255:red, 255; green, 0; blue, 0 }  ,draw opacity=1 ]   (357.5,62) .. controls (344.83,106.85) and (313.14,132.69) .. (289.32,140.44) ;
\draw [shift={(287.5,141)}, rotate = 343.74] [color={rgb, 255:red, 255; green, 0; blue, 0 }  ,draw opacity=1 ][line width=0.75]    (10.93,-3.29) .. controls (6.95,-1.4) and (3.31,-0.3) .. (0,0) .. controls (3.31,0.3) and (6.95,1.4) .. (10.93,3.29)   ;
%Curve Lines [id:da9483538278089707] 
\draw [color={rgb, 255:red, 255; green, 1; blue, 1 }  ,draw opacity=1 ]   (223.5,105) .. controls (223.5,121.24) and (242.65,131.09) .. (254.83,137.17) ;
\draw [shift={(256.5,138)}, rotate = 206.57] [color={rgb, 255:red, 255; green, 1; blue, 1 }  ,draw opacity=1 ][line width=0.75]    (10.93,-3.29) .. controls (6.95,-1.4) and (3.31,-0.3) .. (0,0) .. controls (3.31,0.3) and (6.95,1.4) .. (10.93,3.29)   ;
%Curve Lines [id:da18578695580717697] 
\draw [color={rgb, 255:red, 253; green, 0; blue, 0 }  ,draw opacity=1 ]   (349.5,123) .. controls (338.83,138.52) and (325.34,137.1) .. (305.37,139.74) ;
\draw [shift={(303.5,140)}, rotate = 351.87] [color={rgb, 255:red, 253; green, 0; blue, 0 }  ,draw opacity=1 ][line width=0.75]    (10.93,-3.29) .. controls (6.95,-1.4) and (3.31,-0.3) .. (0,0) .. controls (3.31,0.3) and (6.95,1.4) .. (10.93,3.29)   ;
%Curve Lines [id:da6322720175581147] 
\draw    (219,217.67) .. controls (224.33,202.67) and (232.67,187.33) .. (237.5,182) ;

% Text Node
\draw (278.71,153.43) node [anchor=north west][inner sep=0.75pt]   [align=left] {$\displaystyle p$};
% Text Node
\draw (488,112) node [anchor=north west][inner sep=0.75pt]   [align=left] {$\displaystyle Z$};
% Text Node
\draw (223.33,2.33) node [anchor=north west][inner sep=0.75pt]  [font=\footnotesize] [align=left] {{\footnotesize Transverse Stable Manifold}};
% Text Node
\draw (239.5,185) node [anchor=north west][inner sep=0.75pt]  [font=\footnotesize] [align=left] {{\footnotesize  Unstable Manifold}};

\end{tikzpicture}

	\caption{Example of case $\Hess H(p)<0$ and $g(p)<0$, so there is a transverse 2-dimensional stable manifold containing infinitely many escape orbits, which are colored red.}\label{fig:figure}
\end{figure}

When the transverse invariant manifold is of dimension two, all of the orbits lying within it (of which there are infinitely many) are escape orbits with limit point $p$ (see Figure \ref{fig:figure}). A transverse invariant manifold of dimension one guarantees exactly two escape orbits (one on each side of $Z$) with limit point $p$. 

Let $C_k$ be the number of critical points of $H$ of index $k$ on $Z$ and $b_k$ the $k$-th Betti number. We will use  the Morse inequality
\begin{equation}\label{eq:morse1}
	C_k \geq b_k(Z)	
\end{equation}
to conclude the proof.

\emph{Case $b_1(Z) > 0$.} In this case there is at least one critical point of $H$ of index one (in fact, at least two, because the first Betti number is even), so there is a saddle point and therefore infinitely many escape orbits.

\emph{Case $b_1(Z) = 0$.} This corresponds to $Z$ consisting of $N\geq 1$ disjoint surfaces all diffeomorphic to $\mathbb S^2$. In this case, there are least two escape orbits for each critical point (one escape orbit on each side of the corresponding sphere), some of which may coincide to form singular periodic orbits. In any case, since there are at least $2N$ critical points there must be at least $2N$ distinct escape orbits. Note that it can still be that the exceptional Hamiltonian has a saddle point on $Z$, in which case there would be infinitely many escape orbits.

Finally, notice than being Morse and the sign conditions presented above are open in the $C^\infty$-topology, so we conclude that the set of $b$-forms for which the theorem applies is not only dense, but also open. This completes the proof.
\end{proof}

We conclude this section with a series of comments.

\begin{remark}
   As already mentioned in the introduction, in the context of $b$-Reeb fields the above result is significantly stronger than the main result stated in \cite{MOP22}: In the above result, we just need to do a $C^\infty$-small perturbation of the $b$-contact form to ensure that the exceptional Hamiltonian is a Morse function, whereas in \cite{MOP22}, it was assumed that the $b$-contact form admits a compatible asymptotically exact $b$-metric. We remark that a $b$-contact form does not generally admit an asymptotically exact $b$-metric. Indeed,
    if the $b$-contact form has a positive exceptional Hamiltonian, then it does not admit a compatible asymptotically exact $b$-metric because the exceptional Hamiltonian cannot be an eigenfunction of the Laplacian for any metric (as non-constant eigenfunctions always have zero mean).
\end{remark}

\begin{remark}\label{rem:boundary}
    When the critical set $Z$ has some components which are on the boundary of $M$, for the case $b_1(Z)=0$ one loses an escape orbit at each critical point on the boundary (the escape orbit that would be on the outside of $M$). Point $(ii)$ of Theorem \ref{thm:main} would then read \textit{at least $N_1 + 2N_2$ escape orbits if $b_1(Z) = 0$, where $N_1$ and $N_2$ are the number of connected components of $Z$ on the boundary and the interior of $M$ respectively}.
\end{remark}

\begin{remark}
We can also be a bit more precise in counting the number of escape orbits of Theorem \ref{thm:main} if we introduce multiplicities to distinguish singular periodic orbits from escape orbits which are not singular periodic orbits. Indeed, in the proof of Theorem \ref{thm:main}, we observed that, if there are no saddle points, then there are exactly two escape orbits for each critical point and exactly two critical points for each of the $N$ connected components of $Z$ (if there are more, there must be at least one saddle point). However, since some of these escape orbits can coincide to form singular periodic orbits, we can only guarantee the existence of $2N$ \textit{distinct} escape orbits. This motivates introducing the following definition.
\begin{definition}
    A \textit{one-way escape orbit} is an escape orbit which is not a singular periodic orbit.
\end{definition}
Now, if in the proof of Theorem \ref{thm:main} we count one-way escape orbits with multiplicity one and singular periodic orbits with multiplicity two, each connected component of the critical set $Z$ will contribute exactly four escape orbits (with multiplicity) if there are no saddle points, as shown in Figure \ref{fig:multiplicities}, or else infinitely many if there is a saddle point. We remark that there cannot exist homoclinic orbits at the non-hyperbolic critical points because of the center manifold theorem. Indeed, the dynamics on the center manifold (the critical set $Z$) is Hamiltonian, so it is periodic in a neighborhood of any critical point $p$ with Morse index $0$ or $2$; then the only way to approach $p$ is tangentially to the stable (or unstable) one-dimensional manifold. We then conclude that there is only one orbit (on each side of $Z$) whose $\omega$-limit (or $\alpha$-limit) is $p$, which prevents from the existence of homoclinic trajectories. Notice that such homoclinic orbits may exist associated to the saddle critical points on $Z$.

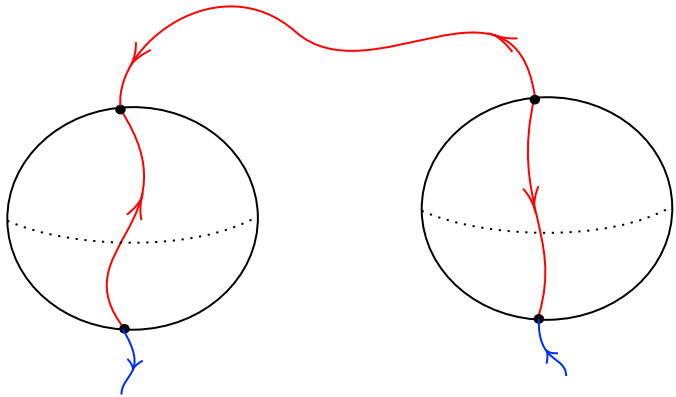
\begin{figure}[H]
    \centering
    \tikzset{every picture/.style={line width=0.75pt}} %set default line width to 0.75pt        

\begin{tikzpicture}[x=0.75pt,y=0.75pt,yscale=-1,xscale=1]

%Curve Lines [id:da04937123041389002] 
\draw [color={rgb, 255:red, 255; green, 0; blue, 0 }  ,draw opacity=1 ]   (174.71,100.26) .. controls (212.77,158.28) and (144.51,169.77) .. (176.72,210.95) ;
\draw  [color={rgb, 255:red, 255; green, 0; blue, 0 }  ,draw opacity=1 ] (177.95,153.29) .. controls (181.26,150.79) and (183.76,147.97) .. (185.43,144.84) .. controls (184.57,148.28) and (184.52,152.05) .. (185.28,156.13) ;
%Shape: Ellipse [id:dp000036216127189181435] 
\draw  [color={rgb, 255:red, 0; green, 0; blue, 0 }  ,draw opacity=1 ] (117.65,156.36) .. controls (117.09,125.34) and (144.92,99.68) .. (179.81,99.05) .. controls (214.7,98.41) and (243.44,123.05) .. (244,154.07) .. controls (244.57,185.08) and (216.74,210.74) .. (181.85,211.38) .. controls (146.96,212.01) and (118.22,187.38) .. (117.65,156.36) -- cycle ;
%Curve Lines [id:da41554953730282107] 
\draw [color={rgb, 255:red, 0; green, 0; blue, 0 }  ,draw opacity=1 ] [dash pattern={on 0.84pt off 2.51pt}]  (117.65,156.36) .. controls (149.51,170.66) and (212.74,172.53) .. (244,154.07) ;

%Shape: Ellipse [id:dp868274772605897] 
\draw  [fill={rgb, 255:red, 0; green, 0; blue, 0 }  ,fill opacity=1 ] (174.58,210.99) .. controls (174.56,209.91) and (175.51,209.02) .. (176.69,209) .. controls (177.87,208.98) and (178.84,209.83) .. (178.86,210.91) .. controls (178.88,211.99) and (177.94,212.88) .. (176.76,212.9) .. controls (175.58,212.92) and (174.6,212.07) .. (174.58,210.99) -- cycle ;
%Shape: Ellipse [id:dp7264010970593733] 
\draw  [fill={rgb, 255:red, 0; green, 0; blue, 0 }  ,fill opacity=1 ] (172.57,100.3) .. controls (172.55,99.22) and (173.5,98.33) .. (174.68,98.31) .. controls (175.86,98.29) and (176.84,99.14) .. (176.85,100.22) .. controls (176.87,101.3) and (175.93,102.19) .. (174.75,102.21) .. controls (173.57,102.23) and (172.59,101.38) .. (172.57,100.3) -- cycle ;
%Curve Lines [id:da4824540404799307] 
\draw [color={rgb, 255:red, 255; green, 0; blue, 4 }  ,draw opacity=1 ]   (174.68,98.31) .. controls (172.15,67.21) and (224.83,27.34) .. (263,61) .. controls (301.17,94.66) and (373,23.75) .. (383.68,93.31) ;
\draw  [color={rgb, 255:red, 255; green, 0; blue, 4 }  ,draw opacity=1 ] (190.02,70.21) .. controls (185.92,72.15) and (182.65,74.53) .. (180.22,77.35) .. controls (181.82,74.09) and (182.58,70.39) .. (182.5,66.25) ;
%Curve Lines [id:da7517107813276127] 
\draw [color={rgb, 255:red, 0; green, 42; blue, 255 }  ,draw opacity=1 ]   (175.82,211.03) .. controls (189.78,234.13) and (175.03,233.59) .. (175.22,244.12) ;
\draw  [color={rgb, 255:red, 0; green, 42; blue, 255 }  ,draw opacity=1 ] (185.97,227.02) .. controls (183.45,228.34) and (181.8,229.77) .. (181.01,231.31) .. controls (180.93,229.67) and (179.99,227.94) .. (178.17,226.11) ;
%Curve Lines [id:da1689349175620638] 
\draw [color={rgb, 255:red, 255; green, 0; blue, 0 }  ,draw opacity=1 ]   (383.68,93.31) .. controls (371,147.25) and (398,163) .. (385.69,204) ;
\draw  [color={rgb, 255:red, 255; green, 0; blue, 0 }  ,draw opacity=1 ] (385.53,138.91) .. controls (383.95,142.75) and (383.24,146.45) .. (383.39,149.99) .. controls (382.38,146.59) and (380.51,143.32) .. (377.78,140.2) ;
%Shape: Ellipse [id:dp8363559202156705] 
\draw  [color={rgb, 255:red, 0; green, 0; blue, 0 }  ,draw opacity=1 ] (326.65,151.36) .. controls (326.09,120.34) and (353.92,94.68) .. (388.81,94.05) .. controls (423.7,93.41) and (452.44,118.05) .. (453,149.07) .. controls (453.57,180.08) and (425.74,205.74) .. (390.85,206.38) .. controls (355.96,207.01) and (327.22,182.38) .. (326.65,151.36) -- cycle ;
%Curve Lines [id:da8883586977146898] 
\draw [color={rgb, 255:red, 0; green, 0; blue, 0 }  ,draw opacity=1 ] [dash pattern={on 0.84pt off 2.51pt}]  (326.65,151.36) .. controls (358.51,165.66) and (421.74,167.53) .. (453,149.07) ;

%Shape: Ellipse [id:dp9506476848285752] 
\draw  [fill={rgb, 255:red, 0; green, 0; blue, 0 }  ,fill opacity=1 ] (383.58,205.99) .. controls (383.56,204.91) and (384.51,204.02) .. (385.69,204) .. controls (386.87,203.98) and (387.84,204.83) .. (387.86,205.91) .. controls (387.88,206.99) and (386.94,207.88) .. (385.76,207.9) .. controls (384.58,207.92) and (383.6,207.07) .. (383.58,205.99) -- cycle ;
%Shape: Ellipse [id:dp4066167437392536] 
\draw  [fill={rgb, 255:red, 0; green, 0; blue, 0 }  ,fill opacity=1 ] (381.57,95.3) .. controls (381.55,94.22) and (382.5,93.33) .. (383.68,93.31) .. controls (384.86,93.29) and (385.84,94.14) .. (385.85,95.22) .. controls (385.87,96.3) and (384.93,97.19) .. (383.75,97.21) .. controls (382.57,97.23) and (381.59,96.38) .. (381.57,95.3) -- cycle ;
%Curve Lines [id:da06927068169702899] 
\draw [color={rgb, 255:red, 0; green, 42; blue, 255 }  ,draw opacity=1 ]   (385.72,205.95) .. controls (385,228.75) and (399.31,224.22) .. (399.5,234.75) ;
\draw  [color={rgb, 255:red, 0; green, 42; blue, 255 }  ,draw opacity=1 ] (389.58,228.53) .. controls (390.23,225.76) and (390.24,223.58) .. (389.6,221.98) .. controls (390.89,222.99) and (392.81,223.42) .. (395.38,223.25) ;
\draw  [color={rgb, 255:red, 255; green, 0; blue, 4 }  ,draw opacity=1 ] (372.36,71.21) .. controls (369.68,67.62) and (366.71,64.8) .. (363.45,62.74) .. controls (367,64.03) and (370.82,64.55) .. (374.93,64.31) ;

\end{tikzpicture}
    \caption{Counting the two one-way escape orbits (in blue) once and the singular periodic orbits (in red) twice, we get a total of eight escape orbits counted with multiplicity. The double multiplicity of the singular periodic orbit connecting the two components is split, each connected component of $Z$ contributing one to the count. This way, each component contributes exactly four escape orbits with multiplicity.}
    \label{fig:multiplicities}
\end{figure}

These considerations allow us to restate Theorem \ref{thm:main} more precisely as follows.

 \begin{theorembis}{thm:main}\label{thm:mainWithMultiplicities}
   Let $\alpha$ be a $b$-contact form on a $3$-dimensional manifold $(M,Z)$ without boundary, with $Z$ a closed embedded surface in $M$.
    Then there exists a $b$-contact form $C^\infty$-close to $\alpha$, such that the associated $b$-Reeb vector field has either
	\begin{enumerate}[(i)]
		\item exactly $4N$ escape orbits counted with multiplicity, where $N$ is the number of connected components of $Z$, or
		\item infinitely many escape orbits.
	\end{enumerate}
 The latter is always the case when $b_1(Z)>0$. Moreover, the set of $b$-contact forms exhibiting these properties is open in the $C^\infty$-topology
\end{theorembis} 

We note that any distribution between one-way escape orbits and singular periodic orbits adding up to $4N$ in the finite case is possible, though this is the subject of future work.
\end{remark}

\begin{remark}
    As a final remark, we note that on any prescribed manifold $M$ it is possible to have at least $2N$ escape orbits for $N$ as large as we want by adding components to the critical set and singularities in the $b$-contact form along these components. Studying topological obstructions to what critical sets are permissible, as well as showing that the bounds given here are sharp, are also some objectives of coming work. 
\end{remark}

\section{Counting escape orbits for $b$-Beltrami fields} \label{sec:escapeBeltrami}

In this final section we prove the second main theorem of this work. To some extent it is rather independent from Theorem~\ref{thm:main} because it holds for any $b$-Beltrami field associated to a generic asymptotically exact $b$-metric. While this has some implications (analogous to Theorem~\ref{thm:main}) in the context of Melrose $b$-contact forms, as discussed in~\cite{MOP22}, its main setting concerns $b$-Beltrami fields, and its proof cannot be approached using elementary Morse theory as in Theorem~\ref{thm:main} because of two reasons:
\begin{itemize}
\item A generic perturbation (even if $C^\infty$-small) of a $b$-Beltrami field for some fixed metric is \emph{no longer} a $b$-Beltrami field for the same metric.
\item The theorem claims that if we fix an asymptotically exact $b$-metric in some residual set of asymptotically exact $b$-metrics, \emph{any} $b$-Beltrami field computed with that metric exhibits the dynamical properties shown below. In particular, the $b$-Beltrami field may have zeros as a section of the $b$-tangent bundle, so it does not yield a $b$-Reeb field. 
\end{itemize}
In summary, although there are similarities and connections between Theorem~\ref{thm:main} and the theorem below, they must be understood as independent results.

\begin{theorem}\label{thm:main2}
	Let $(M,Z)$ be a $3$-dimensional $b$-manifold without boundary, and $Z$ a closed hypersurface in $M$. There exists a residual set $\widehat{\mathcal{G}^k_b}\subset\mathcal{G}^k_b$ in the space of asymptotically exact $b$-metrics such that any $b$-Beltrami vector field on $(M,Z,g)$ with $g\in\widehat{\mathcal{G}^k_b}$, which is not identically zero on $Z$, has either
	\begin{enumerate}[(i)]
		\item infinitely many escape orbits if $b_1(Z)>0$, or
		\item at least $2N$ escape orbits if $b_1(Z) = 0$, where $N$ is the number of connected components of $Z$. 
	\end{enumerate}
\end{theorem}
Here $b_1(Z)$ denotes the first Betti number of the critical surface $Z$ and we use \textit{residual} as in Theorem \ref{thm:hammorse} and Remark \ref{rem:residual}: a countable intersection of open and dense subsets of $\mathcal{G}^k_b$. The case when $M$ has boundary is discussed in Remark \ref{rem:boundary} below.

\begin{proof}
We follow \cite{MOP22} and a linear analysis similar to the one done in the proof of  Theorem \ref{thm:main}. For the sake of completeness, we sketch the proof.

 Let $X$ be a $b$-Beltrami vector field on $(M,Z)$ for an asymptotically exact $b$-metric. Since limit points of escape orbits are critical points of $X$ on $Z$ (here we regard $X$ as a vector field in the usual sense, not as a $b$-vector field), the idea of the proof is to study the stable and unstable manifolds at these points. Since we are interested in the local behaviour around critical points, we use the same local coordinates $(x,y,z)$ introduced in Theorem \ref{thm:hammorse}. By the same theorem, the critical points of $X$ are in fact the critical points of the corresponding exceptional Hamiltonian $f(x,y):=-X_z(x,y,0)$ (see Remark \ref{rem:derivatives}). Let $p=(x_0,y_0,0)$ be a critical point of $f$. We now analyze the linear stability by computing the Jacobian matrix $DX(p)$. Recall from Remark \ref{rem:derivatives} that on $Z$,
    \begin{equation*}\label{eq:derivatives}
        \begin{cases}
            X_x =  -\frac{1}{\lambda\sqrt{\det h}}\frac{\partial f}{\partial y}\\
            X_y =   \frac{1}{\lambda\sqrt{\det h}}\frac{\partial f}{\partial x}.
        \end{cases}
    \end{equation*}
    Therefore, computing $\frac{\partial}{\partial x} X_x(p)$, for example, we obtain
    \begin{align*}\label{matrix}
	\frac{\partial X_x}{\partial x}(p) &= 
	\frac{\partial}{\partial x} \left(\frac{-1}{\lambda\sqrt{\det h}}\frac{\partial f}{\partial y} \right)(p) =
	\frac{\partial}{\partial x} \left(\frac{-1}{\lambda\sqrt{\det h}}\right)\frac{\partial f}{\partial y} (p) + \frac{-1}{\lambda\sqrt{\det h}}\frac{\partial^2 f}{\partial x\partial y} (p) =  \\
    &= -\frac{1}{\lambda\sqrt{\det h}}\frac{\partial^2 f}{\partial x\partial y}(p),
	\end{align*}
    where in the last equality we have used that $p$ is a critical point of $f$. Similarly one computes $\frac{\partial}{\partial y}X_x$,  $\frac{\partial}{\partial x}X_y$ and $\frac{\partial}{\partial y}X_y$ at $p$. For the $z$ component of $X$, we see that $\frac{\partial}{\partial x} (zX_z)(p) = 0$ because $p\in Z = \{z = 0\}$, and likewise $\frac{\partial}{\partial y} (zX_z)(p) = 0$. Finally,
    \begin{equation*}
        \frac{\partial}{\partial z} (zX_z)(p) = X_z(p) = -f(p),
    \end{equation*}
    so bringing everything together we obtain
	\begin{equation*}
	DX(p)= \frac{1}{\lambda\sqrt{\det h}} 
	\left.\begin{pmatrix}
  	-\frac{\partial^2 f}{\partial x\partial y}  & -\frac{\partial^2 f}{\partial y^2} & * \\
  	\frac{\partial^2 f}{\partial x^2} & \frac{\partial^2 f}{\partial x\partial y} & * \\
  	0 & 0 & -\lambda\sqrt{\det h}f
	\end{pmatrix}\right|_p .
	\end{equation*}
    It is easy to determine the linear stability at $p$ by looking at the eigenvalues of this matrix. The eigenvalues are $\lambda_+,\lambda_-$ and $\lambda_z$, where $\lambda_+$ and $\lambda_-$ are eigenvalues of the first $2\times 2$ minor,
    \begin{equation*}
        \lambda_\pm = \pm\frac{\sqrt{-\Hess f(p)}}{\lambda\sqrt{\det h}},
    \end{equation*}
    and $\lambda_z=-f(p)$. 
    
    Assume now as in the statement of the theorem that the asymptotically exact $b$-metric is generic, in the sense that $g\in\widehat{\mathcal{G}^k_b}$ (see also Remark \ref{rem:residual}), and that $X$ is not identically zero on $Z$, so $f$ is not locally constant. Then by Theorem \ref{thm:hammorse}, $f$ is a Morse function with a regular zero set. Therefore, the matrix $DX(p)$ is non-singular and there are stable, unstable and possibly centre manifolds around the critical point which yield escape orbits.
    
    The proof then finishes by the same arguments of the proof of Theorem \ref{thm:main}.

\end{proof}

\begin{remark}\label{rem:boundary}
    As before, when the critical set $Z$ has some components which are on the boundary of $M$, for the case $b_1(Z)=0$ one loses an escape orbit at each critical point on the boundary (the escape orbit that would be on the outside of $M$). Point $(ii)$ of Theorem \ref{thm:main2} would then read \textit{at least $N_1$ + $2N_2$ escape orbits if $b_1(Z) = 0$, where $N_1$ and $N_2$ are the number of connected components of $Z$ on the boundary and the interior of $M$ respectively}.
\end{remark}

%%%%%%%%%%%%%%%%%%%%%%%%%%%%%%%%%%%%%%%%%%%%%%%%%%%%%%%
% REFERENCES

\end{document}